\documentclass[a4paper,12pt]{article}
\usepackage{graphicx} 
\usepackage[
	top=3cm, 
	bottom=2.5cm, 
	inner=3.0cm, 
	outer=3.0cm, 
	footskip=1.4cm, 
	headsep=0.8cm, 
	headheight=0.5cm, 
]{geometry}
\usepackage[utf8]{inputenc} 
\usepackage[T1]{fontenc} 
\usepackage[english]{babel}
\usepackage{amsmath,amsfonts,amssymb,amsthm,amstext,amscd}
\usepackage{bbm}
\usepackage{microtype} 
\usepackage{lmodern}
\usepackage[
	colorlinks=true,
	linkcolor=blue, 
	citecolor=red, 
	urlcolor=magenta, 
    hyperfootnotes=false
]{hyperref}
\usepackage{xcolor}
\usepackage[nottoc]{tocbibind}
\usepackage{tikz}

\usepackage{amscd}
\usepackage{amssymb}
\usepackage{dsfont}
\usepackage{fancyhdr}
\usepackage{enumerate}
\usepackage[utf8]{inputenc}
\usepackage{xcolor}
\usepackage{authblk}
\usepackage{tikz-cd}

\usepackage[figurename=Fig.]{caption}
\usepackage[style=numeric,backend=biber,maxbibnames=99]{biblatex}
\theoremstyle{definition}
\newtheorem{theorem}{Theorem}
\newtheorem{lemma}[theorem]{Lemma}
\newtheorem{corl}[theorem]{Corollary}

\newcommand{\innprod}[2]{\langle #1,#2 \rangle}

\newcommand{\fld}[1]{\mathbb{#1}}
\newcommand{\diag}{\text{diag}}
\newcommand{\Diag}{\text{Diag}}
\newcommand{\allones}{\mathbbm{1}}
\newcommand{\mc}{\text{mc}}
\newcommand{\fcc}{\text{fcc}}

\newcommand{\littletaller}{\mathchoice{\vphantom{\big|}}{}{}{}}
\newcommand\restr[2]{{
  \left.\kern-\nulldelimiterspace 
  #1 
  \littletaller 
  \right|_{#2} 
  }}
\renewcommand{\qedsymbol}{$\blacksquare$}

\renewcommand{\sc}{\textsc}

\begin{document}
\pagestyle{plain}

\begin{center}
\textbf{\Large Semidefinite programming bounds on fractional cut-cover and maximum 2-SAT for highly regular graphs} \\
\vspace{1.5cc}
{ \sc Henrique Assumpção\footnotemark[1],
  Gabriel Coutinho\footnotemark[1]\\
\vspace{0.3cm}
{\small \footnotemark[1]Department of Computer Science, Federal University of Minas Gerais, Brazil}\\
\vspace{0.3cm}
{\small \texttt{[henrique.soares,gabriel]@dcc.ufmg.br}}}
 \end{center}

 \begin{abstract}
We use semidefinite programming to bound the fractional cut-cover parameter of graphs in association schemes in terms of their smallest eigenvalue. We also extend the equality cases of a primal-dual inequality involving the Goemans-Williamson semidefinite program, which approximates \sc{maxcut}, to graphs in certain coherent configurations. Moreover, we obtain spectral bounds for \sc{max 2-sat} when the underlying graphs belong to an association scheme by means of a certain semidefinite program used to approximate quadratic programs, and we further develop this technique in order to explicitly compute the optimum value of its gauge dual in the case of distance-regular graphs.
 \end{abstract}
\section{Introduction}\label{sec:intro}
Let $G = (V,E)$ be a simple, undirected, loopless graph with adjacency matrix $A$. The well-known maximum cut problem (\sc{maxcut}) consists of finding a partition of $V$ into two sets that maximizes the number of edges between them, and we can express this problem as the following quadratic program:
\begin{equation}\label{eq:mc_def}
    \mc(G) := \max\left\{\frac{1}{4} \innprod{L}{xx^T} : x \in \fld{R}^V,x_i^2 = 1\right\},
\end{equation}
where $L$ is the Laplacian matrix of $G$, and $\innprod{\cdot}{\cdot}$ is the trace inner product of the full real matrix algebra $M_V(\fld{R})$ of $|V| \times |V|$ matrices indexed by $V$. The celebrated Goemans-Williamson algorithm \cite{Goemans1995-gl} shows how to approximate \sc{maxcut} within a factor of $\alpha_{\text{GW}}\approx 0.878$ by constructing cuts based on solutions to the semidefinite program (\sc{sdp}) relaxation of \eqref{eq:mc_def} given by:
\begin{equation}\label{eq:eta_def}
    \begin{split}
        \eta(G) :=& \max\left\{\frac{1}{4}\innprod{L}{M} : M \succcurlyeq 0,\diag(M) = \allones \right\}\\
    \end{split}
\end{equation}
where $M \succcurlyeq 0$ denotes that the square matrix $M$ is positive semidefinite, $\allones$ is the all-ones vector, and $\diag$ is the function that maps the diagonal entries of a matrix to a vector.

Much information about $\mc(G)$ and $\eta(G)$ can be obtained by studying other programs that satisfy a type of duality known as antiblocking duality, or more generally as gauge duality. The former has been successfully exploited to study many combinatorial objects throughout the past decades, such as matroids and permutations \cite{Fulkerson1971}, and the latter has been recently employed as a general framework for studying pairs of graphs parameters such as the coclique number, the fractional chromatic number, and the Lovász Theta \cite{proenca2021dual}.

In order to properly define the notion of antiblocking or gauge dual pairs of parameters, it is necessary to view them as functions whose input is a graph together with edge weights $w \in \mathbb{R}_+^E$ (see \cite{proenca2023} for the specific treatment, or \cite[Part III]{rockafellar1997convex} for a general presentation). For the particular case of $\mc(G)$ and $\eta(G)$, we can simply consider the weighted Laplacian
\[
\mathcal{L}(w) := \sum_{ij \in E}w_{ij}(e_i-e_j)(e_i-e_j)^T
\]
which can be seen as an operator mapping edge weights to symmetric matrices. This operator maps nonnegative weights to positive semidefinite matrices, and we note that $L = \mathcal{L}(\allones)$. In this work, however, we focus on the unweighted scenario (all weights constant, or, equivalently, equal to $1$), and therefore we take the shortcut of defining the dual graph parameters as the special case of the dual gauge functions taken on weights everywhere equal to $1$.

The antiblocking dual parameter to \sc{maxcut} is the fractional cut-cover number (\sc{fcc}), which is defined as
\begin{equation}\label{eq:fcc_def}
    \fcc(G) := \min\left\{\allones^Ty : y \in \fld{R}^{\mathcal{P}(V)}_+,\sum_{S \subseteq V}y_S\cdot\allones_{\delta(S)} \geq \allones\right\},
\end{equation}
where $\mathcal{P}(V)$ is the power set of $V$, $\delta(S)$ is the set of edges with one end in $S$ and the other in $V \setminus S$, and $\allones_{\delta(S)}$ is the indicator vector of $\delta(S)$.
The integer solutions for the previous program provide a covering of the edges of the graph by cuts, and the minimum number of such cuts is known as the cut-cover number (\sc{cc}). Computing \sc{cc} arises in certain tests for short circuits in printed circuit boards \cite{LOULOU1992301}, and finding exact or even approximate solutions is known to be NP-hard \cite{TUNGYANG1994193,motwani}. The \sc{fcc} parameter was first introduced by Šámal \cite{SAMAL2005455}, who employed it to study graph homomorphisms and cut-continuous functions. Neto and Ben-Ameur \cite{NETO2019168} later demonstrated how to construct a polynomial-time computable approximation for \sc{fcc} by means of vector colourings, and also showed how this problem is related to frequency allocation. 

Recently, the authors of \cite{proenca2023} have shown how to construct an efficient approximation algorithm for the weighted version of \sc{fcc} with factor $1/\alpha_{\text{GW}} \approx 1.139$, by combining the Goemans-Williamson algorithm with the dual gauge parameter of $\eta(G)$ given by:
\begin{equation}\label{eq:eta_dual_def}
\begin{split}
        \eta^\circ(G):=& \min\{\mu :  \mu \geq 0,N \succcurlyeq 0, \diag(N) = \mu\cdot\allones,\frac{1}{4}\mathcal{L}^*(N) \geq \allones\},
\end{split}
\end{equation}
where $\mathcal{L}^*$ denotes the adjoint of the Laplacian of $G$, and is defined by $\mathcal{L}^*(N)_{ij} = N_{ii} + N_{jj} - 2N_{ij}$ for an edge $ij \in E(G)$. In order to derive the previous expression, we can first consider the SDP dual program to the formulation of $\eta(G)$ in \eqref{eq:eta_def}:
\begin{equation}\label{eq:eta_sdp_dual}
    \eta(G) = \min\left\{\rho  : \rho \geq 0,x \in \fld{R}^V,\rho \geq \allones^Tx,\Diag(x) \succcurlyeq \frac{1}{4}L\right\},
\end{equation}
where $\Diag$ is the function that maps a vector to a diagonal matrix containing its entries, which can be seen as the adjoint of $\diag$. As extensively discussed in \cite[Sec. 3.1.]{proenca2023}, the dual gauge function can be defined analogously to a dual norm as
\[
\eta^\circ(G) = \max\{\allones^Tw:w \in \mathbb{R}^E_+,x \in \mathbb{R}, \Diag(x) -\frac{1}{4}\mathcal{L}(w) \succcurlyeq 0, \allones^Tx \leq 1\},
\]
where we note that the constraints are obtained by requiring a solution to the weighted version of \eqref{eq:eta_sdp_dual} to have objective value at most $1$. The program given in \eqref{eq:eta_dual_def} is promptly obtained by SDP strong duality.

Using the framework of gauge duality, one can show that the pairs $\mc(G),\fcc(G)$ and $\eta(G),\eta^\circ(G)$ satisfy
\begin{subequations}
\begin{alignat}{2}
    |E| &\leq \mc(G)\fcc(G)\label{eq:ineq_mc_comb},\\
    |E| &\leq \eta(G)\eta^\circ(G)\label{eq:ineq_eta_sdp},
\end{alignat}
\end{subequations}
and it can also be shown that equality is attained in both cases for edge-transitive graphs (for a complete discussion of gauge duality and its applications to graph parameters, see \cite{proenca2021dual}, and for a study of the equality cases of similar equations as the above involving the Lovász Theta, see \cite{de_Carli_Silva_2019}).

In this context, we are interested in studying the optimal solutions for the semidefinite programs $\eta(G),\eta^\circ(G)$ when $G$ presents high structural regularity in order to better understand their relationship with \sc{maxcut} and \sc{fcc}. We are also interested in studying the solutions of more general semidefinite programs that can be used to approximate problems such as \sc{max 2-sat}. More specifically, this work achieves the following:
\begin{itemize}
    \item In Sections \ref{sec:mc_cc}, \ref{sec:mc_1wrg}, we extend the equality cases of \eqref{eq:ineq_eta_sdp} for graphs belonging to certain coherent configurations, which include all distance-regular graphs. 
    We combine an observation from \cite{Bachoc2012} about orthogonal projections onto $*$-algebras with the algebraic properties of coherent algebras to construct the proof;
    \item In Section \ref{sec:mc_as}, we calculate $\eta^\circ(G)$ in terms of the smallest eigenvalue of the adjacency matrix of $G$ for graphs belonging to association schemes, which in turn yields a spectral bound to \sc{fcc}. The proof utilizes a key observation from \cite{Goemans1999} that the structure of such graphs allows us to convert certain semidefinite programs into linear programs, enabling us to compute an explicit optimal solution;
    \item In Section \ref{sec:max2sat}, we consider a more general semidefinite program framework used for approximating quadratic programs, first introduced in \cite{Goemans1995-gl} to approximate \sc{max 2-sat}. We compute the optimal value of this semidefinite program when the underlying graphs belong to an association scheme, which in turn yields bounds to \sc{max 2-sat} in terms of the first eigenmatrix of the scheme, and we also compute the optimum of its gauge dual when the scheme arises from a distance-regular graph. We combine the methods developed throughout Section \ref{sec:mc_fcc} with classical facts about association schemes, such as the orthogonality relations and the algebraic properties of the eigenmatrices, to arrive at the desired results.
\end{itemize}

\section{Coherent configurations and association schemes}\label{sec:coherent}
We say that a set $\mathcal{C} = \{A_0,A_1,...,A_d\}$ of square matrices with entries in $\{0,1\}$ is a \textit{coherent configuration} if the following hold:
\begin{enumerate}[(1)]
    \item $\sum_i A_i = J$, where $J$ denotes the matrix with all entries equal to one;
    \item If $A_i \in \mathcal{C}$, then $A_i^T \in \mathcal{C}$;
    \item If $A_i \in \mathcal{C}$ and $A_i$ has a nonzero diagonal entry, then it is a diagonal matrix;
    \item There are nonnegative integers $p_{ij}^l$ such that
    \[
    A_iA_j = \sum_{l = 0}^d p_{ij}^l\cdot A_l,
    \]
    for any indices $0 \leq i,j \leq d$.
\end{enumerate}
The diagonal matrices in $\mathcal{C}$ that partition the identity matrix $I$ are called \textit{fibers}, and if there is only one of these, that is, if $I \in \mathcal{C}$, then we say that the configuration is \textit{homogeneous}. If all matrices in $\mathcal{C}$ commute, we say that the configuration is \textit{commutative}, and in this case it is easy to see that the configuration must also be homogeneous: the only nonzero diagonal matrix with entries in $\{0,1\}$ that commutes with $J$ is the identity matrix. If all matrices in $\mathcal{C}$ are symmetric, then by taking the transpose of the resulting expression for $A_iA_j$ given by item (4), we see that $\mathcal{C}$ must also be commutative, and hence homogeneous. For the remainder of this text, we shall refer to symmetric coherent configurations as \textit{association schemes}, following Cameron~\cite{cameron2003coherent}. We shall also refer to coherent configurations and association schemes simply as configurations and schemes, respectively.

Coherent configurations were initially introduced by Higman~\cite{Higman1975-oi} to study permutation groups, as, for instance, one can always construct a coherent configuration from the orbitals of a permutation group acting on some set, as well as from the Cayley graphs generated by its conjugacy classes. Association schemes on the other hand were first introduced to study statistical experiments, but have since found many important applications in graph theory --- e.g.\ in the study of distance-regular and strongly-regular graphs~\cite{brouwer2011distance} ---, coding theory --- e.g.\ in the study of error-correcting codes~\cite{delsarte1973algebraic} ---, and in many other areas of combinatorics~\cite{bailey2004association}.
In the case of commutative configurations, we get that the linear span of the matrices in $\mathcal{C}$ over $\fld{C}$ forms a commutative $*$-algebra --- that is, an algebra closed under conjugate transposition ---, which in turn implies that there exists a basis of primitive idempotents $E_0 = (1/|V|)J,E_1,\ldots,E_d$ that are the projectors onto the common eigenspaces of the matrices in $\mathcal{C}$ (see for instance~\cite[Ch.~2]{bailey2004association}). The $*$-algebra spanned by the matrices in a configuration $\mathcal{C}$ is known as a \textit{coherent algebra}, and in the case of association schemes we obtain the well-known Bose--Mesner algebra.
The familiar vertex-transitive and edge-transitive graphs are intimately related to coherent algebras. If $\Gamma = \operatorname{Aut}(G)$ denotes the automorphism group of a graph $G$, then the set
\[
C(\Gamma) := \{M \in M_V(\fld{C}) : MP = PM,\;\forall P \in \Gamma\},
\]
known as the centralizer algebra (or commutant algebra) of $\Gamma$, is a coherent algebra, where the group $\Gamma$ is identified with its natural representation as a subgroup of permutation matrices. Some structural properties of $G$ correspond to algebraic properties of $C(\Gamma)$: if $G$ is vertex-transitive then $C(\Gamma)$ is always homogeneous; and if $G$ is edge-transitive we have three possibilities: (i) if $G$ is $1$-arc-transitive, then $C(\Gamma)$ will be homogeneous and $A$ will belong to $\mathcal{C}$; (ii) if $G$ is edge-transitive and vertex-transitive but not $1$-arc-transitive, then $C(\Gamma)$ will be homogeneous and $A$ will \textit{split} in $\mathcal{C}$, that is, $A = A_i + A_i^T$ for some index $0 \leq i \leq d$; or (iii) if $G$ is edge-transitive but not vertex-transitive, then it will be bipartite and thus $C(\Gamma)$ will have two fibers and $A$ will split in $\mathcal{C}$ (see~\cite[Ch.~17]{biggs1993algebraic} for more details).

Our analysis relies on one key result about these objects: the orthogonal projection onto a coherent algebra preserves positive semidefiniteness, i.e., if $M$ is positive semidefinite, then the orthogonal projection $M'$ onto a coherent algebra $\mathcal{A}$ is also positive semidefinite (see \cite[Corollary 9.1]{Bachoc2012}). In the next sections, we use this fact to study solutions to \eqref{eq:eta_def} and \eqref{eq:eta_dual_def} that lie in coherent algebras with similar structure as to the ones associated with edge-transitive graphs, that is, we consider graphs whose adjacency matrix $A$ either belongs to or splits in some coherent configuration $\mathcal{C}$.

\section{MAXCUT and FCC}\label{sec:mc_fcc}
\subsection{Graphs in coherent configurations}\label{sec:mc_cc}
In this section we study the equality cases of \eqref{eq:ineq_eta_sdp}, which will prove to be useful for deriving the desired spectral bounds for \sc{fcc} and \sc{max 2-sat}. We begin with an auxiliary result that allows us to assume the existence of optimal solutions for \eqref{eq:eta_def} and \eqref{eq:eta_dual_def} in a given coherent algebra, provided that some technical conditions are satisfied. The existence of such solutions is known when the underlying coherent algebra contains the matrices that define the SDP (see for instance \cite{Goemans1999,deKlerk2011,Bachoc2012,van_Dam_2014}), however in our case this is not true in general and an adaptation is required, as we show below.
\begin{lemma}\label{lemma:proj}
    Let $G$ be a graph with Laplacian matrix $L$, and let $\mathcal{C} = \{A_0,...,A_d\}$ be a coherent configuration with coherent algebra $\mathcal{A}$ such that $L \in \mathcal{A}$. If $M$ is a feasible solution for \eqref{eq:eta_def}, then its orthogonal projection onto the coherent algebra $\mathcal{A}$ generated by $\mathcal{C}$ is also feasible and has the same objective value. The same is true for any feasible solution $N$ of \eqref{eq:eta_dual_def}.
\end{lemma}
\begin{proof}
    Let $M,N$ be feasible solutions to \eqref{eq:eta_def} and \eqref{eq:eta_dual_def}, respectively. By \cite[Corollary 9.1]{Bachoc2012}, it follows that both projections $M'$ and $N'$ on $\mathcal{A}$ are positive semidefinite matrices. As the basis matrices in $\mathcal{C}$ form an orthogonal basis to $\mathcal{A}$, we may write:
    \[
    M' = \sum_{i=0}^d \frac{\innprod{M}{A_i}}{\innprod{A_i}{A_i}}\cdot A_i.
    \]
    Now recall that
    \[
    \innprod{M}{A_i} = \text{tr}(MA_i^T) = \text{sum}(M \circ A_i),
    \]
    where $\circ$ denotes the entrywise matrix product---also known as the Schur product---and \text{sum} denotes the operator that adds all entries of a given matrix. If $A_i$ is a fiber of the configuration, that is, a 01-diagonal matrix, and as the diagonal entries of $M$ are all constant, it follows that $\text{sum}(M \circ A_i)$ equals the number of $1$s in $A_i$'s diagonal, i.e., $\innprod{A_i}{A_i}$. Combining this with the previous expression for $M'$ yields $\diag(M') = \allones$, and the linearity of the operators involved also easily yields that $\diag(N') = \mu\cdot\allones$. Since $L \in \mathcal{A}$ and the orthogonal projection operator is self-adjoint, it follows that
    \[
    \innprod{M'}{L} = \innprod{M}{L'} = \innprod{M}{L},
    \]
    hence $M'$ has the same objective value as $M$.
    
    It remains to show that $\mathcal{L}^*(N') \geq 4 \cdot \allones$. First, since $L \in \mathcal{A}$ and $L$ has a nonzero entry for each edge $ij$ in $G$, it follows that there exists a unique basis matrix $A_l \in \mathcal{C}$ such that $(A_l)_{ij} = 1$, thus by expressing $N'$ similarly to the formula for $M'$ above, we obtain
    \[
    N'_{ij} = \frac{\innprod{N}{A_l}}{\innprod{A_l}{A_l}} = \frac{\sum_{ab:(A_l)_{ab} = 1} N_{ab}}{\innprod{A_l}{A_l}}.
    \]
    As the basis matrices in $\mathcal{C}$ have disjoint support, it follows that all nonzero entries in $A_l$ correspond to edges in $G$, and by the feasibility conditions of \eqref{eq:eta_dual_def}, it holds that, for any edge $ab \in E(G)$,
    \[
    4 \leq \mathcal{L}^*(N)_{ab} = N_{aa} + N_{bb} - 2N_{ab} = 2\mu-2N_{ab}.
    \]
    Thus, for $ij$ edge of $G$,
    \[
    -N'_{ij} = \frac{\sum_{ab:(A_l)_{ab} = 1} -N_{ab}}{\innprod{A_l}{A_l}} \geq 2-\mu.
    \]
    This in turn implies
    \[
    \mathcal{L}^*(N')_{ij} = 2\mu - 2N'_{ij}  \geq 4,
    \]
    and as the objective value of \eqref{eq:eta_dual_def} is given by $\mu$, it follows that $N'$ is feasible and has the same objective value as $N$, which concludes the proof.
\end{proof}
The previous result allows us to assume that optimal solutions for \eqref{eq:eta_def} and \eqref{eq:eta_dual_def} lie in $\mathcal{A}$ when it contains the Laplacian matrix $L$. For the remaining of this section, we will be interested in regular graphs that display similar properties to edge-transitive graphs. It is worth noting that for a $k$-regular graph, since $L = kI-A$, the condition that $L$ belongs to a certain coherent algebra is equivalent to requiring that $A$ also does. With these remarks, we can now conclude a type of equivalence between the feasibility conditions of the aforementioned programs.
\begin{lemma}\label{lemma:equiv_eta}
    Let $G$ be a graph with adjacency matrix $A$ that either belongs to or splits in a coherent configuration $\mathcal{C}$, and let $M \in \mathcal{A}$. Then for any $\mu > 0$, $N = \mu\cdot M$ is a feasible solution for \eqref{eq:eta_dual_def} if, and only if, $M$ is a feasible solution for \eqref{eq:eta_def} with objective value at least $|E|/\mu$.
\end{lemma}
\begin{proof}
Let $L$ be the Laplacian of $G$, $M \in \mathcal{A}$ and $\mu > 0$. Note that if $A$ either belongs to or splits in $\mathcal{C}$, then the entries of $M$ will be constant on the edges of $G$. This is immediate in the former case, and as for the latter case it suffices to note that if $A = A_i + A_i^T$ for some $A_i \in \mathcal{C}$, then
\[
M_{ab} = \frac{\innprod{M}{A_i}}{\innprod{A_i}{A_i}} = \frac{\innprod{M}{A_i^T}}{\innprod{A_i^T}{A_i^T}} = M_{ba},
\]
for any edge $ab \in E(G)$ such that $(A_i)_{ab} = 1$ --- and similarly for the edges in $A_i^T$. Now, by noting that $L$ can be seen as the weighted Laplacian function evaluated at $\allones$, we get
\begin{equation}\label{eq:temp_eq}
\frac{1}{4}\innprod{L}{M} = \frac{1}{4}\innprod{\allones}{\mathcal{L}^*(M)} = \frac{1}{4}\sum_{ij \in E(G)}\mathcal{L}^*(M)_{ij} = \frac{\mathcal{L}^*(M)_{ij}\cdot|E|}{4},
\end{equation}
for any $ij \in E(G)$, where the last equality follows from the previous remarks. If $N := \mu \cdot M$ is feasible for \eqref{eq:eta_dual_def}, then it easily follows that $M \succcurlyeq 0$ and $\diag(M) = \allones$, and since $4 \leq \mathcal{L}^*(N)_{ij} = \mu \mathcal{L}^*(M)_{ij}$ for any $ij \in E(G)$, this combined with \eqref{eq:temp_eq} yields that the objective value of $M$ is at least $|E|/\mu$. Conversely, if $M$ is feasible for \eqref{eq:eta_def} with objective value at least $|E|/\mu$, again it easily follows that $N \succcurlyeq 0,\diag(N)= \mu\cdot\allones$, and combining the fact that all entries of $M$ are constant on the edges of $G$ with \eqref{eq:temp_eq}, we get that $\mathcal{L}^*(N)_{ij} \geq 4$ for any $ij \in E(G)$, which concludes the proof. 
\end{proof}
Combining the two lemmas, we prove the main result of this section:
\begin{theorem}\label{thm:eta_coherent_algebra}
    If $G$ is a graph whose adjacency matrix either belongs to or splits in a coherent configuration $\mathcal{C}$, then
    \[
    \eta(G)\eta^\circ(G) = |E|.
    \]
\end{theorem}
\begin{proof}
    Let $M$ be an optimal solution for \eqref{eq:eta_def}, which by Lemma \ref{lemma:proj} can be assumed to belong to the coherent algebra generated by $\mathcal{C}$. We can thus take $\mu = |E|/\eta > 0$ and consider the matrix $N = \mu\cdot M$. As $M$ is feasible with objective value $\eta = |E|/\mu$, we can apply Lemma \ref{lemma:equiv_eta} to conclude that $N$ is feasible for \eqref{eq:eta_dual_def}, with objective value $|E|/\eta$ by construction. Since \eqref{eq:eta_dual_def} is a minimization problem, we get
    \[
    \eta(G)\eta^\circ(G) \leq |E|,
    \]
    which together with \eqref{eq:ineq_eta_sdp} implies the desired result.
\end{proof}
We recall that if $G$ is edge-transitive, and if $\Gamma$ is its automorphism group, then the centralizer $C(\Gamma)$ is a coherent algebra whose coherent configuration either contains or splits $A$. In the case of distance-regular graphs, their Bose-Mesner algebra clearly contains $A$ in its association scheme. Thus, we obtain the following corollary:
\begin{corl}
    If $G$ is either an edge-transitive or a distance-regular graph, then \eqref{eq:ineq_eta_sdp} is an equality.

    \hfill \qedsymbol
\end{corl}

\subsection{1-walk-regular graphs}\label{sec:mc_1wrg}
In this section, we briefly discuss a family of graphs that also attains equality in \eqref{eq:ineq_eta_sdp}, even though their adjacency matrices do not belong to or split in coherent configurations in general. 

A graph $G$ with adjacency matrix $A$ is said to be walk-regular if, for any nonnegative integer $l$, the number of closed walks of length $l$ on any vertex is constant, i.e., $A^l$ has constant diagonal. From this definition, we can easily see that both vertex-transitive and distance-regular graphs are walk-regular, and also that, more generally, any graph whose adjacency matrix belongs to a homogeneous coherent algebra is walk-regular. These graphs were originally introduced by Godsil and McKay \cite{GODSIL198051}, who also notably proved that their characteristic polynomial $\Phi_G(x)$ is invariant with respect to vertex removal, i.e., $\Phi_{G-v}(x) = \Phi_{G-u}(x)$ for any $u,v \in V(G)$. Rowlinson \cite{rowlinson} found a way of characterizing distance-regular graphs in terms of walks, showing that such graphs are defined by the property that the number of walks between any two vertices is a constant that depends only on the size of the walk and the distance between the vertices.  

Equipped with the previous definitions, one can see walk-regularity and distance-regularity as particular instances of $k$-walk-regularity, introduced by Dalfó, Fiol and Garriga \cite{Dalfo2009-yp}. We say that a graph of diameter $d$ is $k$-walk-regular, with $0 \leq k \leq d$, if the number of walks of length $l$ between any two vertices at distance $i$ is a constant that only depends on $l$ and $i$, given that $0 \leq i \leq k, l \geq 0$. Hence, the case where $k = 0$ corresponds to walk-regular graphs, and $k = d$ to distance-regular graphs.

Now, we focus on the case of $1$-walk-regular graphs, that intuitively generalize graphs that are both vertex-transitive and edge-transitive. In this case, the definition implies that for any nonnegative integer $l$, there exists constants $a_l,b_l$ such that
\[
A^l \circ I = a_l\cdot I\quad\text{and}\quad A^l \circ A = b_l\cdot A.
\]
Let $\mathcal{A}$ be the algebra generated by the adjacency matrix $A$ of such a graph --- called the adjacency algebra. This is a $*$-algebra, and due to $1$-walk-regularity, we can construct an orthogonal basis $\{I,A,A_2,...,A_d\}$ such that for $l \geq 2$ we have $A_l \circ I = 0$ and $A_l \circ A = 0$, i.e., all matrices projected onto $\mathcal{A}$ will have constant diagonal entries and be constant on the edges of $A$. Combining this with the fact that the Laplacian matrix will certainly belong to $\mathcal{A}$ allows us to conclude similar versions of Lemmas \ref{lemma:proj} and \ref{lemma:equiv_eta}, which promptly implies the following:
\begin{corl}
    If $G$ is a $1$-walk-regular graph, then \eqref{eq:ineq_eta_sdp} is an equality.

    \hfill \qedsymbol
\end{corl}

\subsection{Association schemes}\label{sec:mc_as}
If we consider a $k$-regular graph $G$ with adjacency matrix $A$, it can be easily shown that
\[
\eta(G) \leq \frac{|V|}{4}(k - \lambda_{\text{min}}(A))\quad\text{and}\quad \eta^\circ(G) \geq \frac{2k}{k-\lambda_{\text{min}}(A)}.
\]
The first bound can be obtained by choosing an appropriate feasible solution to \eqref{eq:eta_def}, and by combining this with \eqref{eq:ineq_eta_sdp} one can promptly obtain the lower bound to $\eta^\circ$. By further requiring that $A$ belongs to an association scheme, Goemans and Rendl \cite{Goemans1999} proved that $\eta(G)$ is always equal to the aforementioned upper bound. In this section, we will show that, under the same assumptions, $\eta^\circ$ is also always equal to the above lower bound.

Throughout the remainder of this section, we assume that $A$ belongs to an association scheme $\mathcal{S} = \{I,A,A_2,...,A_d\}$ with Bose-Mesner algebra $\mathcal{A}$ and primitive idempotents $E_0,...,E_d$, and we let $P,Q \in M_{d+1}(\fld{R})$ be the first and second eigenmatrices of $\mathcal{S}$, respectively, that is, the matrices defined by the equations
\begin{equation}
A_i = \sum_{l=0}^d P_{li}\cdot E_l\quad\text{and}\quad E_i = \frac{1}{|V|}\sum_{l=0}^d Q_{li}\cdot A_l. \label{eq:A_i and E_l}	
\end{equation}
We note that $PQ = QP = |V|\cdot I$, and that these are both real matrices as $\mathcal{S}$ is a set of symmetric matrices. We denote the degrees of the scheme by $k_i = p_{ii}^0$, hence each matrix $A_i$ can be seen as the adjacency matrix of a $k_i$-regular graph, and we further note that $P_{l0} = Q_{l0} = 1$, $P_{0l} = k_l,Q_{0l} = m_l$,$P_{ji}m_j=Q_{ij}k_i$, for all indices $0 \leq i,j,l \leq d$, where $m_l = \text{tr}(E_l)$ (see \cite[Ch.2]{bailey2004association} for more details). From these definitions one can obtain the well-known orthogonality relations:
\[
\sum_{l=0}^d \frac{P_{il}P_{jl}}{k_l} = 
        \begin{cases}
            \frac{|V|}{m_i}&\quad\text{if }i = j,\\
            0&\quad\text{otherwise},
        \end{cases}\quad\text{and}\quad \sum_{l=0}^d P_{li}P_{lj}m_l = 
        \begin{cases}
            |V|k_i&\quad\text{if }i = j,\\
            0&\quad\text{otherwise}.
        \end{cases}
\]
From \eqref{eq:A_i and E_l}, the eigenvalues of any linear combination $\sum_i x_i\cdot A_i$ of the basis matrices of $\mathcal{A}$ are given by the linear combinations of the respective eigenvalues of the $A_i$, hence the constraint $M \in \mathcal{A},M \succcurlyeq 0$ is equivalent to requiring that $Px \geq 0$, where $x \in \fld{R}^{d+1}$ is the vector of coefficients such that $M = \sum_{i}x_i\cdot A_i$. With these observations, we can prove the following:

\begin{theorem}\label{thm:eta_dual_as}
    If $G$ is a $k$-regular graph whose adjacency matrix $A$ belongs to an association scheme, and if $\lambda_{\text{min}}(A)$ is its smallest eigenvalue, then
    \[
    \eta^\circ(G) = \frac{2k}{k-\lambda_{\text{min}}(A)}.
    \]
    In particular, we have
    \[
     1 \leq \frac{\fcc(G)}{\frac{2k}{k-\lambda_{\text{min}}(A)}} \leq  \frac{1}{\alpha_{\text{GW}}}.
    \]
\end{theorem}
\begin{proof}
    The previous comments about association schemes combined with Lemma \ref{lemma:proj} allows us to rewrite the semidefinite program $\eta^\circ(G)$ as the linear program
    \[
    \eta^\circ(G) = \min\left\{x_0 : x \in \fld{R}^{d+1},Px \geq 0, x_0 \geq 0, x_0 - x_1 \geq 2\right\},
    \]
    where $x = (x_0,...,x_d)$, and the constraint $x_0 - x_1 \geq 2$ follows from the constraint $\mathcal{L}^*(N) \geq 4\cdot\allones$ and the fact that $\mathcal{L}^*(N)_{ij} = 2x_0 - 2x_1$ if we assume that $N \in \mathcal{A}$. We can then use linear program strong duality and some straightforward algebraic manipulation to obtain:
    \[
    \eta^\circ(G) = \max\left\{2b : y \in \fld{R}_+^{d+1},a,b \in \fld{R}_+, P^Ty = b\cdot e_1 + (1-a-b)\cdot e_0 \right\},
    \]
    where $e_i$ is the $i$-th canonical basis vector in $\fld{R}^{d+1}$, with indexing starting at $0$. Noting that the inverse of $P$ is $(1/|V|)\cdot Q$, that $P_{ji}m_j = Q_{ij}k_i$ and that the first row of $Q$ contains the traces of the respective primitive idempotents, we get
    \[
    y = \frac{1}{|V|}\sum_{l=0}^d\left(\left(\frac{bP_{l1}}{k} + 1-a-b\right)m_l\right)\cdot e_l,
    \]
    thus $y$ is completely determined by our choice of $a,b$. We can also note that if $(y,a,b)$ is feasible for our program, we can take $a'=0$ and $y' = y + (a/|V|)\cdot Q^Te_0$, which is nonnegative since $Q_{0i} = m_i \geq 0$ for any index $0 \leq i \leq d$, and thus $(y',a'=0,b)$ is a feasible solution with the same objective value. So we may always assume that an optimal solution has $a = 0$. Combining this with the fact that $P_{01} = k$ and that each entry of $y$ must be nonnegative, we get that
    \[
     b \leq \frac{k}{k-P_{l1}},
    \]
    for all indices $1 \leq l \leq d$. As we wish to maximize $b$, we set
    \[
    b = \underset{1 \leq l \leq d}{\min}\left\{\frac{k}{k-P_{l1}}\right\} = \frac{k}{k-\lambda_{\text{min}}(A)},
    \]
    and with this the result follows immediately.
\end{proof}
We remark that one could also obtain the previous result by combining Theorem \ref{thm:eta_coherent_algebra} with the aforementioned result from \cite{Goemans1999} that $\eta(G) = (|V|/4)(k-\lambda_{\text{min}}(A))$ for graphs in association schemes. However, the constructive proof given above highlights some methods that will be particularly useful in the next section.

\section{Quadratic programs and MAX 2-SAT}\label{sec:max2sat}
We now study a broader framework for approximating quadratic programs via semidefinite programs, that generalizes the results for \sc{maxcut} derived in the previous sections. Let $G_1 = (V,E_1),G_2 = (V,E_2)$ be two graphs on the same vertex set $V$, and let $L$ be the Laplacian matrix of $G_1$ and $K$ be the signless Laplacian matrix of $G_2$. We can define the quadratic program
\begin{equation}\label{eq:qp_def}
    \begin{split}
        \text{qp}(G_1,G_2) &:= \max\left\{\sum_{ij \in E_1}(1-x_ix_j) + \sum_{ij \in E_2}(1 + x_ix_j): x \in \fld{R}^V,x_i^2 = 1 \right\} \\
        &= \max\left\{\left\langle \frac{L + K}{2} , xx^T \right\rangle: x \in \fld{R}^V,x_i^2 = 1 \right\},
    \end{split}
\end{equation}
with semidefinite relaxation given by
\begin{equation}\label{eq:gamma_def}
    \gamma(G_1,G_2) := \max\left\{\left\langle \frac{L + K}{2} , M\right\rangle : M \succcurlyeq 0,\diag(M) = \allones\right\}.
\end{equation}
These programs can be seen as generalizations of \sc{maxcut} and its semidefinite relaxation defined in equations \eqref{eq:mc_def} and \eqref{eq:eta_def}, respectively, since
\[
\text{mc}(G_1) = \frac{1}{2}\text{qp}(G_1,G_2)\quad\text{and}\quad\eta(G_1) = \frac{1}{2}\gamma(G_1,G_2),
\]
if we set $E_2 = \emptyset$. The above programs can also be easily extended to allow for edge weights in $G_1$ and $G_2$, and in \cite{Goemans1995-gl}, Goemans and Williamson show how to use solutions of $\gamma(G_1,G_2)$ to construct an $\alpha_{\text{GW}}$-approximation algorithm for $\text{qp}(G_1,G_2)$.

Notably, \eqref{eq:qp_def} can be used to model \sc{max 2-sat}. Recall that an instance of the \sc{max 2-sat} problem consists of a collection of Boolean clauses $C_1,...,C_m$, where each clause is a disjunction of at most two literals drawn from a set of variables $\{z_1,...,z_n\}$, with a literal being either a variable or its negation. The problem asks then to assign truth values to the variables in order to maximize the number of satisfied clauses. This is a classic problem that is known to be NP-hard \cite{garey}. The approximation ratio of $\approx 0.878$ given in \cite{Goemans1995-gl} has been improved to $\approx 0.94$ in general \cite{Lewin}, with ratios as high as $\approx 0.954$ attained for certain restricted versions of the problem \cite{austrin,brakensiek2023tightapproximabilitymax2sat}.

Following \cite[Sec. 7.2.1]{Goemans1995-gl}, one can model \sc{max 2-sat} as a quadratic program by introducing $n+1$ variables $\{x_0,...,x_n\}$ with $x_i \in \{\pm1\}$, where $x_0$ is an auxiliary variable that encodes a truth assignment by setting $z_i$ to true if $x_i = x_0$, and false otherwise, for $i \in \{1,...,n\}$. The value $v(C)$ of a clause $C$ is defined to be $1$ if it is true and $0$ otherwise, and as clauses can have at most two literals, it is easy to manually check that the value of any clause can be expressed as a linear combination of terms of the form $1+x_ix_j$ and $1-x_ix_j$. The \sc{max 2-sat} problem can then be written as
\[
\begin{split}
    &\max\left\{\sum_{j=1}^m v(C_j):x_i^2 = 1,i \in \{1,...,n\} \right\}\\
    =&\max\left\{\sum_{i,j \in \{0,...,n\},i \neq j} \alpha_{ij}(1-x_ix_j) + \beta_{ij}(1+x_ix_j):x_i^2 = 1,i \in \{1,...,n\} \right\},
\end{split}
\]
for some nonnegative coefficients $\alpha_{ij},\beta_{ij}$. The graphs $G_1,G_2$ can thus be constructed with vertex set given by the variables $x_i$ and weighted edges defined by $\alpha_{ij},\beta_{ij}$, respectively, and similarly to what was done for \textsc{maxcut}, we study the instances in which all nonzero edge weights are constant. 

Analogously to \eqref{eq:eta_dual_def}, define the program
\begin{align}\label{eq:gamma_dual_def}
    \gamma^\circ(G_1,G_2) := \min\{\mu : \  & \mu \geq 0,N \succcurlyeq 0,  \diag(N) = \mu\cdot\allones, \nonumber \\ & \frac{1}{2}\mathcal{L}^*(N) \geq \allones,\frac{1}{2}\mathcal{K}^*(N) \geq \allones\},
\end{align}
where $\mathcal{K}^*$ is a function that maps a positive semidefinite matrix $N$ to a vector in $\fld{R}_+^{E_2}$ such that $\mathcal{K}^*(N)_{ij} = N_{ii} + N_{jj} + 2N_{ij}$ --- that is, the adjoint of the signless Laplacian function that maps nonnegative edge weights to the usual weighted signless Laplacian matrix. The parameters $\gamma,\gamma^\circ$ form a primal-dual pair of gauges (when specializing to the unweighted scenario, as discussed in the introduction) satisfying the inequality
\begin{equation}\label{eq:gamma_ineq}
    |E_1|+|E_2| \leq \gamma(G_1,G_2)\gamma^\circ(G_1,G_2).
\end{equation}
Similarly to what was done in the previous sections, we inquire into the optimal solutions of these semidefinite programs for graphs with high structural regularity, and also discuss the equality cases of \eqref{eq:gamma_ineq}.

From now on, we assume that the adjacency matrices $A_1,A_2$ of $G_1,G_2$ are elements of an association scheme $S = \{I,A_1,A_2,...,A_d\}$ with Bose-Mesner algebra $\mathcal{A}$, and we use the same notation as in Section \ref{sec:mc_as} to refer to the primitive idempotents and eigenmatrices of $\mathcal{S}$. We first note that the feasibility region of \eqref{eq:gamma_def} is the same as \eqref{eq:eta_def}, hence any feasible solution $M$ may be assumed to belong to $\mathcal{A}$ by an analogous version of Lemma \ref{lemma:proj}.
We can apply a similar argument as in Section \ref{sec:mc_as} to transform $\gamma(G_1,G_2)$ into a linear program: the constraint $\Diag(M) = \allones$ implies that if $M \in \mathcal{A}$ then $M = I + \sum_{l=1}^dx_i\cdot A_i$, hence the constraint $M \succcurlyeq 0$ is equivalent to $Rx \geq 0$, where $x \in \fld{R}^{d}$ and $R$ is the $(d+1) \times d$ matrix obtained by deleting the first column of $P$. Noting that $L = k_1\cdot I - A_1,Q = k_2\cdot I+A_2$, it follows that
\[
\innprod{\frac{L + K}{2}}{M} = \frac{1}{2}(\innprod{k_1\cdot I-A_1}{M} + \innprod{k_2\cdot I+A_2}{M}) = \frac{|V|}{2}((k_1+k_2) + (k_2x_2 - k_1x_1)),
\]
and with this we can rewrite \eqref{eq:gamma_def} as the following linear program:
\begin{equation}\label{eq:gamma_lp}
    \begin{split}
        \gamma(G_1,G_2) &= \frac{|V|}{2}((k_1 + k_2) + \max\{(k_2\cdot e_2 - k_1\cdot e_1)^Tx : x \in \fld{R}^d,Rx \geq -\allones\})\\
        &= \frac{|V|}{2}((k_1 + k_2) + \min\{\allones^Ty : y \in \fld{R}^{d+1}_+,R^Ty = k_1\cdot e_1 - k_2\cdot e_2\}),
    \end{split}
\end{equation}
where the second equality follows from strong duality. As a consequence, we explicitly compute $\gamma(G_1,G_2)$:
\begin{theorem}\label{thm:gamma}
    If $G_1,G_2$ are graphs whose adjacency matrices $A_1,A_2$ belong to an association scheme with first eigenmatrix $P$, then
    \[
    \gamma(G_1,G_2) = \frac{|V|}{2}  \left((k_1 + k_2) + \underset{0 \leq l \leq d}{\max}\{P_{l2} - P_{l1}\}\right).
    \]
\end{theorem}
\begin{proof}
    We let $z$ denote the first row of the second eigenmatrix $Q$, that is, we have that ${z = (1,m_1,...,m_d)}$ is a vector containing the multiplicities of the scheme, and moreover, since $QP = |V|\cdot I$, and $R$ is obtained from deleting the first column of $P$, it follows that $R^Tz = 0$. As $P$ is invertible, we get that $R^T$ has rank $d$, which implies that $\ker(R^T) = \langle z \rangle$. On the other hand, we may use the orthogonality relations to explicitly construct a solution for the system $R^Ty = k_1\cdot e_1 - k_2\cdot e_2$ by considering
    \[
    y' = \frac{1}{|V|}\sum_{l=0}^d ((P_{l1}-P_{l2})m_l)\cdot e_l,
    \]
    and then noting that
    \[
    (R^Ty')_i = \frac{1}{|V|}\left(\sum_{l=0}^d P_{l1}P_{li}m_l - \sum_{l=0}^d P_{l2}P_{li}m_l \right) = \begin{cases}
        k_1,&\quad\text{if }i=1,\\
        -k_2,&\quad\text{if }i=2,\\
        0,&\quad\text{otherwise}.\\
    \end{cases}
    \]
    This implies that any solution for $R^Ty = k_1\cdot e_1 - k_2\cdot e_2$ is of the form $y = y' + (\alpha/|V|) \cdot z$, for any real $\alpha$, but note also that $\allones^Ty' = 0$ by the orthogonality relations, and $\allones^Tz = \sum_{l=0}^dm_l = |V|$. Hence, the objective value of any feasible $y$ is precisely $\alpha$. We then need to choose the minimum $\alpha$ that makes $y$ feasible, which is given by
    \[
    \alpha^* = \underset{0 \leq l \leq d}{\max}\{P_{l2} - P_{l1}\},
    \]
    since the entries of $y$ must all be nonnegative, and this together with \eqref{eq:gamma_lp} concludes the proof.
\end{proof}
As previously mentioned, one can construct an $\alpha_{\text{GW}}$-approximation algorithm for $\text{qp}(G_1,G_2)$ with the solutions of $\gamma(G_1,G_2)$, which combined with the previous theorem allows us to bound $\text{qp}(G_1,G_2)$ in terms of the entries of the first eigenmatrix of the scheme. In particular, this also allows us to bound the value of \textsc{max 2-sat} for the instances discussed at the beginning of this section.
\begin{corl}
    If $G_1,G_2$ are graphs whose adjacency matrices $A_1,A_2$ belong to an association scheme with first eigenmatrix $P$, then
    \[
    \alpha_{\text{GW}} \leq \frac{\text{qp}(G_1,G_2)}{\frac{|V|}{2}\cdot\left((k_1 + k_2) + \underset{0 \leq l \leq d}{\max}\{P_{l2} - P_{l1}\}\right)} \leq 1.
    \]
    \hfill \qedsymbol
\end{corl}
We turn our attention to $\gamma^\circ(G_1,G_2)$, and in order to compute its optimal value analytically, we shall restrict ourselves to association schemes coming from distance-regular graphs. First, we recall that if a distance-regular graph $G$ with diameter $d$ has intersection array $\iota(G) = \{b_0,...,b_{d-1};c_1,...,c_d\}$, and if $P$ is the first eigenmatrix of the scheme generated by $G$, then
\begin{equation}\label{eq:drg_eig}
P_{l2} = \frac{k_2}{b_1k_1}(P_{l1}^2 - (k_1-b_1-1) P_{l1} - k_1),
\end{equation}
for any index $0 \leq l \leq d$ (see \cite[Ch.4]{brouwer2011distance} for more details). With this, we prove the following:
\begin{theorem}\label{thm:gamma_dual}
    Let $G_1$ be a distance-regular graph with diameter $d$ and let $G_2$ be its distance-$2$ graph, with respective adjacency matrices $A_1$ and $A_2$. If $P$ is the first eigenmatrix associated with the association scheme generated by $A_1$, then
    \[
    \gamma^\circ(G_1,G_2) = \begin{cases}
        \frac{k_1}{k_1 - P_{d1}},&\quad\text{if }k_2P_{d1} + k_1P_{d2} > 0,\\
        \frac{k_1}{k_1 - P_{d1}} - \frac{(k_2P_{d1} + k_1P_{d2})}{2k_2(k_1-P_{d1})},&\quad\text{otherwise.}
    \end{cases}
    \]
\end{theorem}
\begin{proof}
    Similarly to what was done in Theorem \ref{thm:eta_dual_as}, we can rewrite $\gamma^\circ(G_1,G_2)$ as the following linear program:
    \[
    \gamma^\circ(G_1,G_2) = \min\{x_0 : x \in \fld{R}^{d+1}, Px \geq 0, x_0 - x_1 \geq 1, x_0 + x_2 \geq 1\},
    \]
    where the last two constraints follow from the fact that any feasible $N \in \mathcal{A}$ for \eqref{eq:gamma_dual_def} is constant on the edges of $G_1$ and $G_2$. By strong duality, we have
    \begin{equation}\label{eq:gamma_dual_lp}
    \begin{split}
        \gamma^\circ(G_1,G_2) = \max\{&b+c : y \in \fld{R}^{d+1}_+,a,b,c \in \fld{R}_+,\\&P^Ty = b\cdot e_1 - c\cdot e_2 + (1-a-b-c)\cdot e_0\},
    \end{split}
    \end{equation}
    and since $|V|^{-1}\cdot Q$ is the inverse of $P$ and $P_{ji}m_j = Q_{ij}k_i$, we get that
    \[
    y = \frac{1}{|V|}\sum_{l=0}^d \left(\left(\frac{bP_{l1}}{k_1} - \frac{cP_{l2}}{k_2} + 1-a-b-c\right)m_l\right)\cdot e_l,
    \]
    and so $y$ is determined by our choice of $a,b,c$. Similarly to the proof of Theorem \ref{thm:eta_dual_as}, we may always assume that any optimal solution $(y,a,b,c)$ is such that $a = 0$, and since each entry of $y$ must be nonnegative, we get that $0 \leq c \leq 1/2$, and that
    \[
    b \leq \underset{1 \leq l \leq d}{\min}\left\{\left(\frac{k_1}{k_1 - P_{l1}} \right) \left(\frac{(1-c)k_2 - cP_{l2}}{k_2} \right) \right\}.
    \]
    If we fix some feasible $c$, as we wish to maximize $b+c$, we may always assume that the previous inequality is an equality, and thus any solution to our program is in fact uniquely determined by our choice of $c$. A simple algebraic manipulation then shows that we can rewrite $\gamma^\circ(G_1,G_2)$ as
    \begin{equation}\label{eq:thm_gamma_dual_middle}
    \gamma^\circ(G_1,G_2) = \underset{0 \leq c \leq 1/2}{\max}\left\{\underset{1 \leq l \leq d}{\min}\left\{\frac{k_1}{k_1 - P_{l1}} - c\frac{(k_1P_{l2} + k_2P_{l1})}{(k_1-P_{l1})k_2} \right\}\right\}.
     \end{equation}
     We now claim that for any fixed $c$, the index $l$ that achieves the minimum in the previous equation is $d$, and we shall prove this by constructing a chain of equivalent inequalities. Indeed, for any index $1 \leq l \leq d$, define
     \[
     \gamma^\circ_l(c) = \frac{k_1}{k_1 - P_{l1}} - c\frac{(k_1P_{l2} + k_2P_{l1})}{(k_1-P_{l1})k_2},
     \]
     and assume w.l.o.g. that the matrix $P$ is ordered such that $P_{d1} < \ldots < P_{01}$, that is, we order the distinct eigenvalues of $G_1$ in decreasing order (and note that there are precisely $d+1$ of those as $G_1$  is distance-regular with diameter $d$). Now consider the indices $1 \leq j \leq i \leq d$, and note that it suffices to show that $\gamma^\circ_i(c) \leq \gamma^\circ_j(c)$ for a fixed $0 \leq c \leq 1/2$. A straightforward computation shows that this is equivalent to requiring that
     \[
     P_{j2}(k_1-P_{i1}) - P_{i2}(k_1-P_{j1}) \leq k_2(P_{j1}-P_{i1}).
     \]
     By substituting the terms $P_{j2}$ and $P_{i2}$ with the respective expressions given in \eqref{eq:drg_eig} and noting that by assumption $P_{j1} - P_{i1} > 0$, we get that the previous inequality is equivalent to $P_{i1}(k_1-P_{j1}) \leq k_1(k_1-P_{j1})$, which holds since $k_1$ is the largest eigenvalue of $G_1$ and $j>0$. By combining this fact with \eqref{eq:thm_gamma_dual_middle}, we obtain
     \[
    \gamma^\circ(G_1,G_2) = \underset{0 \leq c \leq 1/2}{\max}\left\{\frac{k_1}{k_1 - P_{d1}} - c\left(\frac{k_2P_{d1} + k_1P_{d2}}{k_2(k_1-P_{d1})}\right) \right\}.
     \]
     If $k_2P_{d1} + k_1P_{d2} > 0$, we take $c = 0$ in order to maximize the objective value, and otherwise we take $c = 1/2$, which concludes the proof.
\end{proof}
We remark that one could further reduce the expression obtained in the previous theorem for $\gamma^\circ$ by substituting $P_{d2}$ according to \eqref{eq:drg_eig}, which would result in a term solely dependent on $P_{d1}$. This however results in a rather obscure expression that does not seem to present any further insight into the problem, and so we opted to omit it.

Since the semidefinite programs for $\gamma$ and $\gamma^\circ$ generalize $\eta$ and $\eta^\circ$, a natural question to ask is whether the equality cases of \eqref{eq:gamma_ineq} behave similarly to the ones of \eqref{eq:ineq_eta_sdp}. We have conducted computational experiments using SageMath \cite{sage} to explicitly check if \eqref{eq:gamma_ineq} is an equality for a given pair of graphs $G_1,G_2$. In particular, we checked all strongly-regular graphs (SRGs) contained in SageMath's implementation of Brouwer's SRG database \cite{Cohen_2016} with at most $1300$ vertices. Of all $1058$ SRGs in question, only $148$ attained equality in \eqref{eq:gamma_ineq}, with the remaining $910$ attaining a strict inequality. Notably, however, we managed to find large examples of SRGs for both cases, suggesting that equality in \eqref{eq:gamma_ineq} might hold for more than a finite set of graphs. The smallest SRG where \eqref{eq:gamma_ineq} is a strict inequality is the Paley graph of parameter $9$, obtained by first considering a finite field $\fld{F}_{9}$ with $9$ elements which will be the vertices of our graph, and then connecting two vertices if their difference is a square in $\fld{F}_{9}$. This is an arc-transitive graph and also a strongly-regular graph, hence its adjacency matrix $A$ belongs to an association scheme $\{I,A,J-A-I\}$, with first eigenmatrix given by
\[
P = \begin{pmatrix}
    1 & 4 & 4\\
    1 & 1 & -2\\
    1 & -2 & 1
\end{pmatrix}.
\]
If we then let $G_1,G_2$ be such graph and its complement, respectively, one can easily see that the vector $y = (0,1/2,0,0,1/2,1/4)$ is a feasible solution for the maximization program in \eqref{eq:gamma_dual_lp} with objective value $3/4 > 8/11 = (|E_1|+|E_2|)/\gamma$, implying that $\gamma(G_1,G_2)\gamma^\circ(G_1,G_2)>|E_1|+|E_2|$.

\section{Conclusion}

In this work, we discussed semidefinite programs used to approximate certain NP-hard graph parameters, and showed that they can be explicitly computed in terms of the eigenvalues when the underlying graphs present enough structural regularity. These computations yielded spectral bounds for \sc{fcc} and \sc{max 2-sat}, and also provided new insights into the relationship between these semidefinite programs and their combinatorial counterparts.

Here are some possible future research directions. We believe that the results from Sections \ref{sec:mc_cc} and \ref{sec:mc_1wrg} can provide further insight into the equality cases of \eqref{eq:ineq_mc_comb}, which in turn could lead to a deeper understanding of the relationship between \sc{maxcut} and \sc{fcc}. The example discussed in Section \ref{sec:max2sat} suggests that stronger conditions on the underlying graphs of \sc{max 2-sat} may be required to attain equality in \eqref{eq:gamma_ineq}, and future inquiry into these conditions could provide new interesting families of graphs and association schemes. Determining the combinatorial structure of the gauge dual of \sc{max 2-sat} may also lead to interesting results, and investigating whether one can construct an approximation algorithm for this parameter with $\gamma^\circ$ could lead to spectral bounds by means of the results in the previous section.

\section*{Acknowledgements}

Both authors acknowledge the support from FAPEMIG and CNPq.

\printbibliography

\end{document}